\documentclass[12pt]{article}
\usepackage[margin=1.25in]{geometry}

\usepackage{mdwlist}
\usepackage{enumerate}

\usepackage{amsmath,amssymb}
\usepackage{graphics, subfigure, float}
\usepackage{fp, calc}
\usepackage{hyperref}
\usepackage{url}

\usepackage[T1]{fontenc} % IMPORTANT: Load before the fonts
\usepackage{fourier}
\usepackage{bm}

\usepackage{amscd,amsthm}

\usepackage{pst-all}
\usepackage{pstricks-add}
\usepackage{pst-func}
\newpsobject{showgrid}{psgrid}{subgriddiv=1,griddots=10,gridlabels=6pt}
\usepackage{verbatim, comment}
\usepackage{datetime}

\newtheoremstyle{theorem}{1em}{1em}{\slshape}{0pt}{\bfseries}{.}{ }{}
\theoremstyle{theorem}
\newtheorem{theorem}{Theorem}

\newtheorem*{theorem*}{Theorem}
\newtheorem{corollary}[theorem]{Corollary}

\newtheorem{lemma}[theorem]{Lemma}

\newtheorem*{claim*}{Claim}

\theoremstyle{remark}
\newtheorem{remark}{Remark}
\newtheorem*{remark*}{Remark}

\newtheoremstyle{example}{1em}{1em}{}{0pt}{\bfseries}{.}{ }{}

\providecommand{\setN}{\mathbb{N}}
\providecommand{\setZ}{\mathbb{Z}}

\providecommand{\setR}{\mathbb{R}}

\newcommand{\inn}[2]{\langle {#1}, {#2} \rangle}
\newcommand{\E}{\mathop{\mathbb{E}}}

\psset{linewidth=1pt, arrowsize=6pt}

\usepackage{calrsfs} % now e.g. \pazocal{I} gives a nice I
\DeclareMathAlphabet{\pazocal}{OMS}{zplm}{m}{n}

\usepackage[displaymath,textmath,graphics, subfigure, floats]{preview} %sections,
\PreviewEnvironment{center} 

\makeatother

\title{Approximate Carath\'eodory bounds via Discrepancy Theory}
\date{}
\author{Victor Reis \thanks{University of Washington, Seattle. Email: {\tt voreis@uw.edu}.} \and Thomas Rothvoss  \thanks{University of Washington, Seattle. Email: {\tt rothvoss@uw.edu}. 
Supported by NSF CAREER grant 1651861 and a David \& Lucile Packard Foundation Fellowship.}}

\begin{document}
\maketitle
\begin{abstract}
  The approximate Carath\'eodory problem in general form is as follows: Given two symmetric convex
  bodies $P,Q \subseteq \setR^m$, a parameter $k \in \setN$ and $\bm{z} \in \textrm{conv}(X)$ with $X \subseteq P$, find $\bm{v}_1,\ldots,\bm{v}_k \in X$ so that $\|\bm{z} - \frac{1}{k}\sum_{i=1}^k \bm{v}_i\|_Q$ is minimized.
  Maurey showed that if both $P$ and $Q$ coincide with the $\| \cdot \|_p$-ball, then an error of $O(\sqrt{p/k})$ is possible.

  We prove a reduction to the vector balancing constant from discrepancy theory which for most cases can provide tight bounds for general $P$ and $Q$. For the case where $P$ and $Q$ are both $\| \cdot \|_p$-balls we prove an upper bound of $\sqrt{ \frac{\min\{ p, \log (\frac{2m}{k}) \}}{k}}$. Interestingly, this bound cannot be obtained taking independent random samples; instead we use the Lovett-Meka random walk. We also prove an extension to the more general case where  $P$ and $Q$ are $\|\cdot \|_p$ and $\| \cdot \|_q$-balls with $2 \leq p \leq q \leq \infty$. %, we provide a bound of
%  \[
%    \frac{1}{\frac{1}{2}-\frac{1}{p}+\frac{1}{q}} \cdot \frac{\sqrt{\min\{ p, \log (\frac{2m}{k}) \}}}{k^{1/2+1/p-1/q}}
%  \]
\end{abstract}

\section{Introduction}

The (exact) Carath\'eodory Theorem is part of most introductory courses on the theory of linear programming: given any vector $\bm{z} \in \textrm{conv}(X)$ where $X \subseteq \setR^m$, there is a subset of 
points $X' \subseteq X$ with $|X'| \leq m+1$ so that already $\bm{z} \in \textrm{conv}(X')$. More recently, the \emph{approximate} version gained interest, where only $k$ vectors from $X$ may be selected with uniform weights and the
goal is to minimize the error in a given norm.

Barman~\cite{BarmanApproximateNash-STOC2015} used an approximate Carath\'eodory bound for algorithms to compute approximate Nash equilibria for bimatrix games as well as for finding $k$-densest subgraphs.
The core argument of \cite{BarmanApproximateNash-STOC2015} is as follows: if one has two players with $n$ strategies each and some payoff matrix $\bm{A} \in [-1,1]^{n \times n}$,
then for any mixed strategy $\bm{y}$ of the column player (i.e. $\bm{y} \in \setR_{\geq 0}^n$ and $\|\bm{y}\|_1=1$) one can apply the approximate Carath\'eodory Theorem for norm $\| \cdot \|_{\infty}$ (or rather equivalently for $\| \cdot \|_{\log(n)}$) and find $k := \Theta(\frac{\log(n)}{\varepsilon^2})$ columns $\bm{a}_1,\ldots,\bm{a}_k$ of $\bm{A}$ so that $\|\bm{A}\bm{y} - \frac{1}{k}\sum_{i=1}^k \bm{a}_i\|_{\infty} \leq \varepsilon$. In other words, any mixed strategy can be $\varepsilon$-approximated by the
unweighted average of only $\Theta(\frac{\log(n)}{\varepsilon^2})$ many pure strategies which then allows for an efficient enumeration.
The approximate Carath\'eodory Theorem has also been useful in algebraic settings. 
For example Deligkas et al~\cite{ApproxPolynomialEqAndIneqJCSS2022} use it to find approximate solutions to systems of polynomial (in)equalities
and Bhargava, Saraf and Volkovich~\cite{FactoringSparsePolynomialsJACM2020} use approximate Carath\'eodory to prove that sparse polynomials have only sparse factors which then allows efficient deterministic factorization of sparse polynomials;
both applications use the variant with respect to the $\| \cdot \|_{\infty}$-norm.

To make the statements formal, for symmetric convex bodies $P,Q \subseteq \setR^m$ and $k \in \setN$, we denote
\[
  \textrm{ac}_k(P,Q) := \sup_{\substack{X \subseteq P, \\ \bm{z} \in \textrm{conv}(X)}} \inf_{\bm{v}_1,\ldots,\bm{v}_k \in X} \Big\|\bm{z}-\frac{1}{k}\sum_{i=1}^k \bm{v}_i\Big\|_Q
\]
as the best error bound with respect to the $\| \cdot \|_Q$-norm for approximating a point $\bm{z}$
in the convex hull of some points in $P$. We would like to point out that the vectors $\bm{v}_1,\ldots,\bm{v}_k$ may be taken with repetition. Here, $\| \cdot \|_Q$ is the norm with $\|\bm{x}\|_Q = \min\{ s \geq 0 \mid \bm{x} \in sQ\}$.
A folklore result is that for the Euclidean norm one has $\textrm{ac}_k(B_2^m,B_2^m) \leq \frac{1}{\sqrt{k}}$ for any $k \geq 1$, which gives a \emph{dimension free} bound. More generally, for $p \geq 1$, it is true that $\textrm{ac}_k(B_p^m,B_p^m) \leq O(\sqrt{\frac{p}{k}})$ where $B_p^m := \{ \bm{x} \in \setR^m \mid \|\bm{x}\|_p \leq 1\}$ is the $\| \cdot \|_p$-unit ball.
This bound is derived from
Maurey's Lemma from functional analysis (which was reported by Pisier~\cite{PisierRemarquesSurUnResultatNonPublicDeMaurey1980}; for an English version, see the appendix of Bourgain and Nelson~\cite{BourgainNelsonArxiv13}).
Algorithmically, the result is simple: write $\bm{z} = \sum_{i=1}^{N} \lambda_i\bm{u}_i$ where $\lambda_1,\ldots,\lambda_N \geq 0$
and $\sum_{i=1}^N \lambda_i=1$. Then sample $\bm{v}_1,\ldots,\bm{v}_k \in \{ \bm{u}_1,\ldots,\bm{u}_N\}$ independently according to  the probabilities $\lambda_i$ (possibly with repetition).

Another approach in the literature by Mirrokni et al~\cite{MirrokniLVW_ICML2017} is based on the desire to avoid the computation of
$\bm{z} = \sum_{i=1}^{N} \lambda_i\bm{u}_i$. Instead they use the \emph{Mirror Descent} algorithm from
convex optimization to compute the sequence $\bm{v}_1,\ldots,\bm{v}_k$ directly. In fact they reprove the
bound of  $\textrm{ac}_k(B_p^m,B_p^m) \leq O(\sqrt{\frac{p}{k}})$ using their framework. More recently,
Combettes and Pokutta~\cite{ApproxCaraViaFrankWolfeMathProg2021} show that the Frank-Wolfe algorithm can also be used to recover the same bounds.
From the current state of the literature, there are two directions that appear natural to follow:
\begin{itemize}
\item \emph{Approximate Carath\'eodory for general pairs of norms.} The existing bounds are for the case where $P$ and $Q$ are the same $\| \cdot \|_p$-ball. Is there a convenient framework that can handle general symmetric convex bodies or at least $P = B_p^m$ and $Q = B_q^m$? % of different $\| \cdot \|_p$-norms? 
\item \emph{Tight bounds for approximate Carath\'eodory.} Generally, it is stated that for example
  the bound $\textrm{ac}_k(B_p^m,B_p^m) \leq O(\sqrt{\frac{p}{k}})$ is tight (see e.g. \cite{MirrokniLVW_ICML2017}). But that is only true if one aims for a dimension independent bound. So for which regimes of $m$ vs. $k$ and $p$
  is it possible to improve the bound?
\end{itemize}

A classical area within combinatorics that appears related to these questions is \emph{discrepancy theory}. Let $S_1,\ldots,S_m \subseteq \{ 1,\ldots,n\}$ be a set system over $n$ elements. Then
the goal is to find a bi-coloring $\bm{x} \in \{ -1,1\}^n$ so that the worst imbalance $\max_{i=1,\ldots,m} |\sum_{j \in S_i}x_j|$ is minimized. A seminal result of Spencer~\cite{SixStandardDeviationsSuffice-Spencer1985} says that for $m \geq n$,
the discrepancy is bounded by $O(\sqrt{n \log(\frac{2m}{n})})$. If no element is in more than $t$ sets, then
one can also prove a bound of $2t$, see Beck and Fiala~\cite{IntegerMakingTheorems-BeckFiala81}.
A convex geometry based method by Banaszczyk~\cite{BalancingVectors-Banaszczyk98} shows that for any $\bm{A} \in \setR^{m \times n}$ with column length $\|\bm{A}^j\|_2 \leq 1$ for all $j=1,\ldots,n$ and any symmetric convex body $K \subseteq \setR^m$ with Gaussian measure at least $1/2$ (for example $K = \Theta(\sqrt{\log(m)}) \cdot B_{\infty}^m$ or $K = \Theta(\sqrt{m}) \cdot B_2^m$ work), there is a coloring $\bm{x} \in \{-1,1\}^n$
with $\bm{A}\bm{x} \in 5K$. Interestingly, neither of these cited results of \cite{SixStandardDeviationsSuffice-Spencer1985,IntegerMakingTheorems-BeckFiala81, BalancingVectors-Banaszczyk98} can be obtained by merely taking a uniform random coloring $\bm{x}$.
But for example, the result by Spencer allows for elegant algorithmic proofs. While the first
such algorithm was due to Bansal~\cite{DiscrepancyMinimization-Bansal-FOCS2010}, we focus on the later work of Lovett and Meka~\cite{DiscrepancyMinimization-LovettMekaFOCS12} whose main claim can be paraphrased as follows:
\begin{theorem}[\cite{DiscrepancyMinimization-LovettMekaFOCS12}] \label{thm:LovettMekaRephrased}
  Let $\bm{A} \in \setR^{m \times n}$ and $\bm{x}_0 \in [-1,1]^n$ and let $C_1,C_2>0$ be small enough constants.
  Then there is an efficiently computable distribution $\pazocal{D}(\bm{A},\bm{x}_0)$ with the following properties:
  \begin{enumerate}
  \item[(A)] One has $|\{j \in [n] : x_j \in \{ \pm 1\}\}| \geq \frac{n}{2}$ for all $\bm{x} \sim \pazocal{D}(\bm{A},\bm{x}_0)$.
  \item[(B)] One has $\|\bm{A}(\bm{x}-\bm{x}_0)\|_{\infty} \leq \Delta$ for all  $\bm{x} \sim \pazocal{D}(\bm{A},\bm{x}_0)$ where $\Delta \geq 0$ is
    any parameter satisfying
       $\sum_{i=1}^m \exp( -C_1\frac{\Delta^2}{\|\bm{A}_i\|_2^2}) \leq C_2n$.
  \item[(C)] One has $\E_{\bm{x} \sim \pazocal{D}(\bm{A},\bm{x}_0)}[\bm{x}] = \bm{x}_0$.
  \item[(D)] The random vector $\bm{x}-\bm{x}_0$ is $O(1)$-subgaussian. In particular ${\|\left<\bm{A}_i,\bm{x}-\bm{x}_0\right>\|_{\psi_2} \lesssim \|\bm{A}_i\|_2}$
    for all $i=1,\ldots,m$.
  \end{enumerate}
  The running time to compute a sample is\footnote{Throughout this work, we will use $T(m,n)$ as the best running time to generate a sample from the described distribution. Whenever we state a running time using $T(\cdot ,\cdot)$ we implicitly assume that the function $T$ is non-decreasing in $m$ and $n$ and that $T(m,n) \geq mn$ which corresponds to the input length.}  $T(m,n) \leq O(n^{1+\omega} + n^2m)$.
\end{theorem}
Here we use the notation $A \lesssim B$ if there is a universal constant $C>0$ so that $A \leq C \cdot B$. 
The statement differs in several aspects to the original statement of \cite{DiscrepancyMinimization-LovettMekaFOCS12}.
We discuss and justify the changes in Appendix~\ref{appendix:LovettMeka}.
Intuitively, $\pazocal{D}(\bm{A},\bm{x}_0)$ is simply the outcome of a \emph{Brownian motion} starting at $\bm{x}_0$ that freezes coordinates as soon as they hit $+1$ or $-1$ and it freezes constraints if $|\left<\bm{A}_i,\bm{x}-\bm{x}_0\right>| = \Delta$. For example, if $\bm{A} \in \{ 0,1\}^{m \times n}$ with $m \geq n$ is the
incidence matrix in Spencer's setting, then one may choose $\Delta := \Theta(\sqrt{n\log(\frac{2m}{n})})$ and obtain a partial coloring of discrepancy $\Delta$ that colors at least half the coordinates\footnote{One can then obtain a full coloring by iterating the argument $O(\log n)$ times. In the $i$th such iteration (with $i \geq 0$) there are at most $n/2^i$ elements uncolored, so the suffered discrepancy decreases to $O(\sqrt{ (n/2^{i}) \log \frac{2m}{n/2^i} })$. Summing over these terms gives a convergent sum of value $O(\Delta)$.}.

\subsection{Our contribution}

The \emph{vector balancing constant} for two symmetric convex bodies $P,Q \subseteq \setR^m$ is
\[
 \textrm{vb}(P,Q) := \sup\Big\{ \min_{\bm{x} \in \{ -1,1\}^n} \Big\| \sum_{i=1}^n x_i\bm{v}_i \Big\|_Q \mid n \in \setN, \bm{v}_1,\ldots,\bm{v}_n \in P \Big\}
\]
The connection between the approximate Carath\'eodory problem and vector balancing was already discovered by Dadush et al~\cite{BalancingVectorsInAnyNorm-DadushEtAl-FOCS2018}
who proved that for any symmetric convex bodies $P,Q \subseteq \setR^m$ one has
$\textrm{ac}_k(P,Q) \leq \frac{\textrm{vb}(P,Q)}{k}$.
%for any $\bm{z} \in \textrm{conv}(X)$ with $X \subseteq P$ and $k \geq 1$, there are points $\bm{v}_1,\ldots,\bm{v}_k \in X$ so that $\|\bm{z} - \frac{1}{k} \sum_{i=1}^k \bm{v}_i\|_Q \leq \frac{\textrm{vb}(P,Q)}{k}$.
But for example for $P=Q=B_2^m$ one has $\textrm{vb}(B_2^m,B_2^m) = \Theta(\sqrt{m})$ and so the obtained bound is $O(\frac{\sqrt{m}}{k})$,
which is suboptimal if $k \ll m$. Instead, we suggest a reduction to a slight variant of the vector balancing
constant that allows for tight bounds. Let 
\[
 \textrm{vb}_{n}(P,Q) := \sup\Big\{ \min_{\bm{x} \in \{ -1,1\}^n} \Big\| \sum_{i=1}^n x_i\bm{v}_i \Big\|_Q \mid \bm{v}_1,\ldots,\bm{v}_n \in P \Big\}
\]
be the vector balancing constant restricted to exactly $n$ vectors\footnote{We should note that this is the same as asking for \emph{at most $n$} vectors as $\bm{0} \in P$.}.
%Note that we fix the number of vectors to be a parameter $n$ and ask that the coloring is \emph{balanced} in the sense that $\sum_{i=1}^n x_i=0$\footnote{While this restricts the available colorings, it only adds \emph{one} additional constraint which in any of the standard methods does not change the asymptotics of the discrepancy.}.
%To state our running time bounds we will denote $T(m,n)$ as the time to generate one sample from the Lovett-Meka distribution $\pazocal{D}(\bm{A})$ with $\bm{A} \in \setR^{m \times n}$ as defined further down.
We prove the following: %\footnote{In the running time statements we will assume that a vector $\bm{z} \in \textrm{conv}(X)$ is given as a convex combination $\bm{z} = \sum_{i} \lambda_i \bm{v}_i$ of at most $m+1$ vectors from $X$.}:

% \begin{theorem} \label{thm:ReductionApxCaraToVB}
%   Let $P,Q \subseteq \setR^m$ be symmetric convex bodies. Then for any $\bm{z} \in \textrm{conv}(X)$ with $X \subseteq P$ and $k \geq 1$ there are $\bm{v}_1,\ldots,\bm{v}_k \in X$ so that
%   \[
%     \Big\|\bm{z} - \frac{1}{k}\sum_{i=1}^k \bm{v}_i\Big\|_Q \leq 2\sum_{\ell \geq 1} \frac{1}{k2^{\ell}} \cdot \textrm{vb}_{\min\{ k2^{\ell},m+1\}}(P,Q)
%   \]
%   The running time to find those vectors is $O(\log n) \cdot T(m,m+1)$. 

%   \red{Should we have a notation like $\mathrm{ac}_k (P, Q)$ so simplify notation? Don't we need to pay a constant of 2 for the $\min\{k2^\ell, m+1\}$ vs $k2^\ell$?}

% \end{theorem}
\begin{theorem} \label{thm:ReductionApxCaraToVB}
  For any symmetric convex bodies $P,Q \subseteq \setR^m$ and any $k \in \setN$ one has
  \[
    \textrm{ac}_k(P,Q) \leq 4\sum_{\ell \geq 1} \frac{1}{k2^{\ell}} \cdot \textrm{vb}_{k2^{\ell}}(P,Q)
  \]
  The vectors $\bm{v}_1,\ldots,\bm{v}_k$ can be found in time $O(\log m)$ times the time to find a coloring behind $\textrm{vb}_{t}(P,Q)$ where $t \leq m+1$,
  % The running time to find the vectors $\bm{v}_1,\ldots,\bm{v}_k$ is $O(\log n) \cdot T(m,m+1)$
  assuming we are given $\bm{z}$ as convex combination of at most $m+1$ vectors from $X$. 

%  \red{Don't we need to pay a constant of 2 for the $\min\{k2^\ell, m+1\}$ vs $k2^\ell$?}

\end{theorem}
For most bodies the quantity $\textrm{vb}_{k}(P,Q)$ grows sublinear in $k$ and the sum is dominated by the first term,
in which case one has $\textrm{ac}_k(P,Q) \lesssim \frac{1}{k} \cdot \textrm{vb}_{k}(P,Q)$. For example if $P = Q = B_2^m$
one then has $\textrm{vb}_k(B_2^m,B_2^m) = \Theta(\sqrt{k})$ and so one recovers the $O(\frac{1}{\sqrt{k}})$ bound mentioned earlier\footnote{Though not all bodies allow for such a sublinear dependence. For example fix $1 \leq k \leq m/2$. Then one has $\textrm{vb}_k(B_1^m,B_1^m) = \Theta(k)$ and so Theorem~\ref{thm:ReductionApxCaraToVB} provides a suboptimal bound of $\textrm{ac}_k(B_1^m,B_1^m) \leq O(\log \frac{m}{k})$. On the other hand, indeed it is true that $\textrm{ac}_k(B_1^m,B_1^m) \geq \Omega(1)$ meaning that the approximate Carath\'eodory bound does not even improve with $k$ (at least as long as $k \leq m/2$). To see this, consider $X := \{ \bm{e}_1,\ldots,\bm{e}_m\}$ and a target of $\bm{z} := (\frac{1}{m},\ldots,\frac{1}{m})$. Then for any $\bm{v}_1,\ldots,\bm{v}_k \in X$ one has $\|\bm{z}-\frac{1}{k}\sum_{i=1}^k \bm{v}_i\|_1 \geq \Omega(1)$ since any coordinate whose unit vector is not included will contribute $\frac{1}{m}$ to the norm.}.
Also note that by \cite{LSV1986}, for any symmetric convex bodies $P,Q \subseteq \setR^m$ and any $t \in \setN$ one has $\textrm{vb}_t(P,Q) \leq 2 \textrm{vb}_m(P,Q)$, meaning that the worst case is basically attained for $m$ many vectors.
Then the infinite sum in Theorem~\ref{thm:ReductionApxCaraToVB} is
dominated by the first $\log(2m/k)$ terms if $k \leq m$ and the first
term if $k \geq m$.
%\begin{note}
%\red{I suppose it is not tight for all bodies, e.g. when $P = Q = B^m_1$ this is giving $O(\log(m/k))$. But clearly this is always at most $O(1)$.}
% \end{note}

For balancing vectors in $B_p^m$ into $B_q^m$ we prove the following: 
\begin{theorem}\label{thm:VectorBalancingFromBpToBq}
  For $2 \leq p \leq q \leq \infty$ and $n \leq m$ one has
  \[
  \textrm{vb}_{n}(B_p^m,B_q^m) \lesssim \frac{\sqrt{\min\{ p, \log (\frac{2m}{n}) \}}}{\frac{1}{2}-\frac{1}{p}+\frac{1}{q}} \cdot n^{1/2-1/p+1/q}
\]
The time to find the corresponding coloring is $O(\log n) \cdot T(m,n)$.
\end{theorem}
We should point out that the upper bound on $\textrm{vb}_n(B_p^m,B_q^m)$ itself is already proven in~\cite{VecBalancingInLebesqueReisRothvossArxiv2020}. However, that
argument goes via the Gaussian measure of the suitable partial colorings and the only known algorithms
are via convex optimization resulting in large polynomial running times. In contrast, here
we give a streamlined argument that shows that the more efficient Lovett-Meka algorithm can be used to
obtain the same bound rather than relying on convex optimization.

Combining the results above then gives:
\begin{theorem} \label{ApxCaraForLpLq}
  Let $2 \leq p \leq q \le \infty$ and $k \in \setN$. Then % and $X \subseteq B_p^m$. Then for any $\bm{z} \in \textrm{conv}(X)$ and $k \in \setN$, one can find $\bm{v}_1,\ldots,\bm{v}_k \in X$ so that
  \[
   \textrm{ac}_k(B_p^m,B_q^m) \lesssim \frac{1}{\frac{1}{2}-\frac{1}{p}+\frac{1}{q}} \cdot \frac{\sqrt{\min\{ p, \log (\frac{2m}{k}) \}}}{k^{1/2+1/p-1/q}}
 \]
 The vectors $\bm{v}_1,\ldots,\bm{v}_k$ can be found in time $O(\log^2 m) \cdot T(m,m+1) \leq O(m^{1+\omega}\log^2(m))$ assuming we are given $\bm{z}$
 as a convex combination of at most $m+1$ vectors in $X$.
\end{theorem}
To the best of our knowledge this is the first approximate Carath\'eodory bound for pairs of different $\| \cdot \|_p$-norms.
In particular, when $p = q$ this improves upon the $O(\sqrt{\frac{p}{k}})$ bound in ~\cite{MirrokniLVW_ICML2017}
whenever $p \ll \log(\frac{2m}{k})$:

\begin{corollary}
 Let $2 \leq p \leq \infty$ and $k \in \setN$. Then % and $X \subseteq B_p^m$. Then for any $\bm{z} \in \textrm{conv}(X)$ and $k \in \setN$, one can find $\bm{v}_1,\ldots,\bm{v}_k \in X$ so that
  \[
%   \Big\|\bm{z} - \frac{1}{k} \sum_{i=1}^k \bm{v}_i\Big\|_q 
\textrm{ac}_k(B_p^m,B_p^m) \lesssim \sqrt{ \frac{\min\{ p, \log (\frac{2m}{k}) \}}{k}}
\]
The vectors $\bm{v}_1,\ldots,\bm{v}_k$ can be found in time  $O(\log^2 m) \cdot T(m,m+1) \leq O(m^{1+\omega}\log^2(m))$ assuming we are given $\bm{z}$
 as a convex combination of at most $m+1$ vectors in $X$.
\end{corollary}

Finally, we show that the bound in Theorem~\ref{thm:VectorBalancingFromBpToBq} is tight up to a factor of $\frac{1}{2} - \frac{1}{p} + \frac{1}{q}$:

\begin{theorem}\label{thm:VBLowerBound}
Let $2 \le p \le q \le \infty$ and $n \le m \le 2^n$. Then 
  \[
  \textrm{vb}_{n}(B_p^m,B_q^m) \gtrsim \sqrt{\min\Big\{ p, \log \Big(\frac{2m}{n}\Big) \Big\}} \cdot n^{1/2-1/p+1/q}.
\]
\end{theorem}

\section{Preliminaries}

In this section we review a few facts that we later rely on. Let $S^{m-1} := \{ \bm{x} \in \setR^m \mid \|\bm{x}\|_2 = 1\}$ be the sphere.

\paragraph{Convex functions.} Recall the following well known fact:
% T: I think we only use Jensen Inequality for concave functions, so I am removing this one:
%\begin{lemma}[Jensen Inequality for convex functions] \label{lem:JensenForConvex}
%  Let $X$ be any $\setR$-valued random variable and let $F : \setR \to \setR$ be a convex function, then
%  $F(\E[X]) \leq \E[F(X)]$.
%\end{lemma}
\begin{lemma}[Jensen Inequality for convave functions] \label{lem:JensenForConcave}
  Let $X$ be any $\setR$-valued random variable and let $F : \setR \to \setR$ be a concave function, then
  $F(\E[X]) \geq \E[F(X)]$.
\end{lemma}

\paragraph{Estimates on $\| \cdot \|_p$ norms.} It will be useful to understand how the
norm $\|\bm{z}\|_p$ of a vector can change depending on $p \in [1,\infty]$.
\begin{lemma} \label{lem:LpVsLq}
For any $\bm{z} \in \setR^m$ and $1 \leq p \leq q \leq \infty$ one has $\|\bm{z}\|_q \leq \|\bm{z}\|_p \leq m^{1/p-1/q}\|\bm{z}\|_q$.
\end{lemma}

\begin{lemma} \label{lem:LqVsLpAndLinfinity}
  For any $\bm{z} \in \setR^m$ and $1 \leq p \leq q \leq \infty$, we have $\|\bm{z}\|_q^q \leq \|\bm{z}\|_p^p \cdot \|\bm{z}\|_{\infty}^{q-p}$.
\end{lemma}

\paragraph{The subgaussian norm.}
%\paragraph{The Lovett-Meka Random Walk.}

We introduce a concept from probability theory that is extremely useful and convenient
when dealing with random variables that have Gaussian-type tails.
For a random variable $X \in \setR$ we define the \emph{subgaussian norm} as
\[
 \|X\|_{\psi_2} := \inf\Big\{ s > 0 : \E\Big[\exp\Big(\frac{X^2}{s^2}\Big)\Big] \leq 2\Big\}
\]
One may think of $\|X\|_{\psi_2}$ as the minimum number so that the tail of $X$
is dominated by the Gaussian $N(0,\|X\|_{\psi_2}^2)$. For example if $X \sim N(0,t^2)$
then $\|X\|_{\psi_2} = \Theta(t)$ and also if $X \sim \{ -t,t\}$ uniformly, then  $\|X\|_{\psi_2} = \Theta(t)$.
It may not be obvious but indeed $\| \cdot \|_{\psi_2}$ is a norm on the space
of jointly distributed random variables, i.e. $\|t X\|_{\psi_2} = |t| \cdot \|X\|_{\psi_2}$ and $\|X_1 + X_2\|_{\psi_2} \leq \|X_1\|_{\psi_2} + \|X_2\|_{\psi_2}$ for jointly distributed random variables $X_1,X_2$ (even if they are dependent).
We will use in particular the following properties:
\begin{lemma} \label{lem:PropertiesOfSubgaussianNorm}
  The subgaussian norm satisfies the following:
  \begin{enumerate}
  \item[(A)] For any real random variable $X$ and any $p \geq 1$ one has $\E[|X|^p]^{1/p} \lesssim \sqrt{p} \cdot \|X\|_{\psi_2}$.
  \item[(B)] If $X_1,\ldots,X_N$ are independent real mean-zero random variables\footnote{The argument also works in the Martingale setting. Suppose for any condititioning on $X_1,\ldots,X_{i-1}$ one has $\|X_i\|_{\psi_2} \leq L_i$. Then $\|X_1 + \cdots + X_N\|_{\psi_2} \lesssim (\sum_{i=1}^N L_i^2)^{1/2}$.}, then
    \[
    \|X_1 + \cdots + X_N\|_{\psi_2} \lesssim \Big(\sum_{i=1}^N \|X_i\|_{\psi_2}^2\Big)^{1/2}
    \]
  \item[(C)] For $\bm{a} \in \setR^m$ and $\bm{x} \sim S^{m-1}$ uniformly one has $\|\left<\bm{a},\bm{x}\right>\|_{\psi_2} \lesssim \frac{1}{\sqrt{m}} \|\bm{a}\|_2$.
  \item[(D)] For any real random variable $X$ with $\E[X] = 0$ and any $\lambda \geq 0$ one has \[\Pr[|X| \geq \lambda \|X\|_{\psi_2}] \leq 2e^{-C\lambda^2},\] where $C>0$ is a universal constant.
  \end{enumerate}
\end{lemma}
We recommend the excellent textbook of Vershynin~\cite{Vershynin_2018_HighDimProbability} for details.
For the lower bound we will use the following reverse Chernoff bound from~\cite{doi:10.1137/12087222X}:

\begin{lemma}\label{lem:AntiChernoff}
Given independent random variables $x_1, \dots, x_n \sim \{-1,1\}$ and $\lambda \in [3, \sqrt{n}/2]$,
\[\Pr[x_1 + \dots + x_n \ge \lambda\sqrt{n}] \ge \exp(-9\lambda^2/2).\]
\end{lemma}

%\begin{proof}
%  The distribution is the outcome of the Lovett-Meka Brownian motion \emph{conditioned on being successful} (which happens
%  with constant probability).
%  We only discuss $(B)$. We can choose $\lambda_1,\ldots,\lambda_m \geq 0$ so that $\sum_{i=1}^m \exp(-\lambda_i^2/16) \leq \frac{n}{16}$ and we will obtain a vector $\bm{x} \in [-1,1]^n$ so that $|\left<\bm{A}_i,\bm{x}\right>| \leq \lambda_i \|\bm{A}_i\|_2$ for all $i=1,\ldots,m$.
%  In order to get $\Delta = \lambda_i\|\bm{A}_i\|_2$ we choose $\lambda_i := \frac{\Delta}{\|\bm{A}_i\|_2}$. Then the condition simplifies to
%  \[
%   \sum_{i=1}^m \exp\Big( -\frac{\Delta^2}{16\|\bm{A}_i\|_2^2}\Big) \leq \frac{n}{16}
%  \]
%\end{proof}

\section{Reduction from Approximate Carath\'eodory to Vector Balancing}

In this section, we prove the reduction of the approximate Carath\'eodory problem to
vector balancing as stated in Theorem~\ref{thm:ReductionApxCaraToVB}. The idea is to
follow the classical approach of \cite{LSV1986}: begin with an arbitrary convex combination
and round the coefficients bit-by-bit. The same basic approach was also followed by
Dadush et al~\cite{BalancingVectorsInAnyNorm-DadushEtAl-FOCS2018}. We prove an auxiliary lemma
that bounds the error when ``doubling the fractionality''.
\begin{lemma} \label{lem:DoubleFractionalityViaVB}
  Let $P,Q \subseteq \setR^m$ be symmetric convex bodies and let $\delta > 0$. 
  Let $\bm{z} = \sum_{i=1}^n \lambda_i \bm{v}_i$ where $\bm{v}_1,\ldots,\bm{v}_n \in P$ and $\bm{\lambda} \in \delta \setZ_{\geq 0}^n$.
  Then there is a vector $\bm{z}' = \sum_{i=1}^n \lambda_i' \bm{v}_i$ where $\bm{\lambda}' \in 2\delta \setZ_{\geq 0}^n$
  so that $\|\bm{z}-\bm{z}'\|_Q \leq \delta \cdot \textrm{vb}_{n}(P,Q)$ and $\sum_{i=1}^n \lambda_i' \leq \sum_{i=1}^n \lambda_i$.
\end{lemma}
\begin{proof}
  Write $\lambda_i = 2 \delta a_i + \delta b_i$ with $b_i \in \{ 0,1\}$ and $a_i \in \setZ_{\geq 0}$.
  Let $I := \{ i \in [n] \mid b_i = 1\}$. %Note that because $\sum_{i=1}^n \lambda_i = 1$ and $\frac{1}{\delta} \in \setN$ we know that $|I|$ is even.
  Now, let $\bm{x} \in \{ -1,1\}^I$ be the
  coloring so that $\|\sum_{i \in I} x_i \bm{v}_i\|_Q \leq \textrm{vb}_{|I|}(P,Q) \leq \textrm{vb}_{n}(P,Q)$.
  We may assume that $\sum_{i \in I} x_i \leq 0$ --- otherwise replace $\bm{x}$ with $-\bm{x}$.
%  After possibly replacing $\bm{\varepsilon}$ by $-\bm{\varepsilon}$ we may assume that $\sum_{i \in I} \varepsilon_i \leq 0$.
  We may extend the vector to $\bm{x} \in \{ -1,0,1\}^m$ by setting $x_i := 0$ for $i \notin I$.
  We update $\lambda_i' := 2\delta a_i + \delta (1+x_i) b_i \in 2\delta \setZ_{\geq 0}$ for $i \in [n]$.
  Next, we define $\bm{z}' := \sum_{i=1}^n \lambda_i' \bm{v}_i$. Then
  \[
   \|\bm{z}-\bm{z}'\|_Q = \delta \Big\|\sum_{i \in I} x_i \bm{v}_i\Big\|_Q \leq \delta \cdot \textrm{vb}_{n}(P,Q)
 \]
 Note that  $\sum_{i \in I} x_i \leq 0$ implies that  $\sum_{i=1}^n \lambda_i' \leq \sum_{i=1}^n \lambda_i$.
This gives the claim.
\end{proof}

%We let $\textrm{vert}(Q)$ be the set of extreme points of $Q$. 
%\begin{theorem}
%  Let $P,Q \subseteq \setR^n$ be symmetric convex bodies and let $X \subseteq P$. Then for any $\bm{z} \in \textrm{conv}(X)$ and $1 \leq k \leq n$,
%  there are $\bm{v}_1,\ldots,\bm{v}_k \in X$ so that
%  \[
%    \Big\|\bm{x} - \frac{1}{k}\sum_{i=1}^k \bm{v}_i\Big\|_K \leq 4\sum_{\ell \in \{k,2k,4k,8k,\ldots,2n\}} \frac{\textrm{vb}_{\ell}^{\textrm{bal}}(P,Q)}{\ell}
%  \]
%\end{theorem}
Next, we iteratively apply Lemma~\ref{lem:DoubleFractionalityViaVB} to an initial convex combination until the convex coefficients are multiples of $\frac{1}{k}$. We almost obtain the desired claim, just that the number of vectors might be \emph{less} than $k$. 
\begin{lemma} \label{lem:ReductionApxCaraToVBWithAtMostkVec}
  Let $P,Q \subseteq \setR^m$ be symmetric convex bodies. Then for any $\bm{z} \in \textrm{conv}(X)$ with $X \subseteq P$ and $k \in \setN$ there are $s \in \{ 0,\ldots,k\}$ and $\bm{v}_1,\ldots,\bm{v}_s \in X$ so that
  \[
    \Big\|\bm{z} - \frac{1}{k}\sum_{i=1}^s \bm{v}_i\Big\|_Q \leq \sum_{\ell \geq 1} \frac{2}{k2^{\ell}} \cdot \textrm{vb}_{k2^{\ell}}(P,Q)
  \]
  The vectors can be found in time $O(\log m)$ times the time to find the colorings in $\textrm{vb}_{t}(P,Q)$ where $t \leq m+1$.
%  The running time to find those vectors is $O(\log n) \cdot T(m,m+1)$. 
\end{lemma}

%Now we prove Theorem~\ref{thm:ReductionApxCaraToVB} which basically works by iterativly 
\begin{proof}
  Fix a point $\bm{z} \in \textrm{conv}(X)$ where $X \subseteq P$. Then we can write $\bm{z} = \sum_{i=1}^{n} \lambda_i \bm{v}_i$ where  $n \leq m+1$, $\bm{v}_1,\ldots,\bm{v}_n  \in X$, $\lambda_i \geq 0$ for all $i=1,\ldots,m$ and $\sum_{i=1}^m \lambda_i = 1$. Without loss of generality we may assume that $\bm{\lambda} \in \frac{2^{-L}}{k}\setZ_{\geq 0}^n$
  for some $L \in \setN$. We abbreviate $\bm{z}^{(L)} := \bm{z}$. Now suppose for $\ell \in \{ 0,\ldots,L\}$ the current iterate is $\bm{z}^{(\ell)}$ so that $\bm{z}^{(\ell)} = \sum_{i=1}^n \lambda_i^{(\ell)} \bm{v}_i$
  and $\bm{\lambda}^{(\ell)} \in \frac{2^{-\ell}}{k} \setZ_{\geq 0}^n$.

  Then we apply Lemma~\ref{lem:DoubleFractionalityViaVB} to obtain a vector  $\bm{z}^{(\ell-1)} = \sum_{i=1}^n \lambda_i^{(\ell-1)} \bm{v}_i$
 with $\bm{\lambda}^{(\ell-1)} \in \frac{2^{-(\ell-1)}}{k} \setZ_{\geq 0}^n$ and $\sum_{i=1}^n \lambda_i^{(\ell-1)} \le 1$. Using that  $|\textrm{supp}(\bm{\lambda}^{(\ell)})| \leq k2^{\ell}$, the approximation error satisfies 
  \[
  \big\|\bm{z}^{(\ell)} - \bm{z}^{(\ell-1)}\big\|_Q \leq \frac{1}{2^{\ell}k} \cdot \textrm{vb}_{|\textrm{supp}(\bm{\lambda}^{(\ell)})|}(P,Q) \leq \frac{1}{2^{\ell}k} \cdot \textrm{vb}_{k2^{\ell}}(P,Q)
  \]
  Note that the final iterate is of the form $\bm{z}^{(0)} = \sum_{i=1}^n \lambda^{(0)}_i \bm{v}_i$ with $\lambda_i^{(0)} \in \frac{\setZ_{\geq 0}}{k}$ and $\sum_{i=1}^n \lambda_i^{(0)} \leq 1$. Then for $s := k\sum_{i=1}^n \lambda_i^{(0)} \in \{ 0,\ldots,k\}$,
  let $\bm{u}_1,\ldots,\bm{u}_s$ be a list of vectors that contains $\bm{v}_i$ exactly $k \lambda_i^{(0)} \in \setZ_{\geq 0}$ times. 
  Using the triangle inequality we obtain
  \[
   \Big\|\bm{z}-\frac{1}{k}\sum_{i=1}^s \bm{u}_i\Big\|_Q \leq \sum_{\ell=1}^L \|\bm{z}^{(\ell)}-\bm{z}^{(\ell-1)}\|_Q \leq \sum_{\ell \geq 1} \frac{1}{k2^{\ell}} \cdot \textrm{vb}_{k2^{\ell}}(P,Q)
 \]
 Now let us discuss the running time. First, we can choose $L \leq O(\log m)$ while possibly making a rounding error of $\max\{ \|\bm{y}\|_Q: \bm{y} \in P\} \leq \textrm{vb}_1(P,Q)$, which we absorb by paying an extra factor of 2. Then the running time is dominated by
 the time to find the colorings. Note that we call $\textrm{vb}_t(P,Q)$ only $L$ times for
 parameters $t$ with $t \leq m+1$.
\end{proof}

Now we will use the same trick as Dadush et al~\cite{BalancingVectorsInAnyNorm-DadushEtAl-FOCS2018} in order to obtain exactly $k$ vectors, at the expense of a factor 2 in the approximation error: 
\begin{proof}[Proof of Theorem~\ref{thm:ReductionApxCaraToVB}]
Let  $\bm{z} \in \textrm{conv}(X)$ where $X \subseteq P$.
Fix a vector $\bm{u}_0 \in X$ and write  $X' := \{ \bm{u}-\bm{u}_0 \mid \bm{u} \in X \}$. Note that in particular $\bm{0} \in X'$
and $\bm{z} - \bm{u}_0 \in \textrm{conv}(X')$.
% and so $z = \sum_{u \in X} \lambda_u u$ for some convex combination $\lambda$. Then $z-u_0 = \sum_{u \in X} \lambda_u (u-u_0)$ which means that $z - u_0 \in conv(X')$ for . Then we can find vectors $$
%Then $z-u_0 \in conv(X')$.
We apply Lemma~\ref{lem:ReductionApxCaraToVBWithAtMostkVec} and obtain vectors $\bm{v}_1,\ldots,\bm{v}_{s} \in X'$ with $s \leq k$ so that
\[
    \Big\|(\bm{z} -\bm{u}_0) - \frac{1}{k}\sum_{i=1}^s \bm{v}_i\Big\|_Q \leq \sum_{\ell \geq 1} \frac{2}{k2^{\ell}} \cdot \textrm{vb}_{k2^{\ell}}(2P,Q)
  \]
  using that $X' \subseteq 2P$.
Since $\bm{0} \in X'$, we can extend this sequence to a list $\bm{v}_1,\ldots,\bm{v}_k$ of $k$ vectors. Each $\bm{v}_i \in X'$
can be written as $\bm{v}_i = \bm{u}_i - \bm{u}_0$ with $\bm{u}_1,\ldots,\bm{u}_k \in X$. Then
\[
 \Big\|\bm{z} - \frac{1}{k} \sum_{i=1}^k \bm{u}_k\Big\|_Q =  \Big\|(\bm{z} -\bm{u}_0) - \frac{1}{k}\sum_{i=1}^k (\bm{u}_k-\bm{u}_0)\Big\|_Q \leq  4\sum_{\ell \geq 1} \frac{1}{k2^{\ell}} \cdot \textrm{vb}_{k2^{\ell}}(P,Q). \qedhere
  %=  \Big\|(\bm{z} -u_0) - \frac{1}{k}\sum_{i=1}^k (\bm{u}_k-u_0)\Big\|_Q \leq \sum_{\ell \geq 1} \frac{1}{k2^{\ell}} \cdot \textrm{vb}_{\min\{ k2^{\ell},m+1\}}^{\textrm{bal}}(2P,Q)
\]
\end{proof}

\section{Vector balancing from $B_p^m$ to $B_q^m$}
From now on, we focus on the case where $P = B_p^m$ and $Q = B_q^m$.
\subsection{Balancing from $B_p^m$ to $B_p^m$}
Suppose we have a matrix $\bm{A} \in \setR^{m \times n}$ with $\bm{A}^1,\ldots,\bm{A}^n \in B_p^m$.
Then for a uniform random $\bm{x} \in \{ -1,1\}^n$ one has $\E[\|\bm{A}\bm{x}\|_p] \lesssim \sqrt{pn}$. This result
extends to the case where $\bm{x}$ is subgaussian rather than independent.
\begin{lemma} \label{lem:LM-ExpectedLpNorm}
  Let $\bm{A}^1,\ldots,\bm{A}^n \in B_p^m$ for $p \geq 2$ and let $\bm{x}_0 \in [-1,1]^n$.
  Then  \[\E_{\bm{x} \sim \pazocal{D}(\bm{A},\bm{x}_0)}[\|\bm{A}(\bm{x}-\bm{x}_0)\|_p] \lesssim \sqrt{p n}.\]
\end{lemma}
\begin{proof}
  It will be convenient to abbreviate $\bm{B} \in \setR^{m \times n}$ as the matrix with entries $B_{ij} := A_{ij}^2$.
  Then we have
  \begin{eqnarray*}
    \E_{\bm{x} \sim \pazocal{D}(\bm{A},\bm{x}_0)}[\|\bm{A}(\bm{x}-\bm{x}_0)\|_p] %&=& \E_{\bm{x} \sim \pazocal{D}(\bm{A})}\big[\big(\|\bm{A}\bm{x}\|_p^p\big)^{1/p}\big] \\
                                                              &\stackrel{(*)}{\leq}& \E_{\bm{x} \sim \pazocal{D}(\bm{A},\bm{x}_0)}\big[\|\bm{A}(\bm{x}-\bm{x}_0)\|_p^p\big]^{1/p} \\
                                                             &=& \Big(\sum_{i=1}^m \E_{\bm{x} \sim \pazocal{D}(\bm{A},\bm{x}_0)}\big[|\left<\bm{A}_i,\bm{x}-\bm{x}_0\right>|^p\big] \Big)^{1/p} \\ &\stackrel{\textrm{Thm~\ref{thm:LovettMekaRephrased}.(D)+Lem~\ref{lem:PropertiesOfSubgaussianNorm}.(A)}}{\lesssim}& \Big(\sum_{i=1}^m \big(\sqrt{p} \cdot \|\bm{A}_i\|_2\big)^p \Big)^{1/p} \\
                                                             &=& \sqrt{p} \cdot \sum_{i=1}^m \Big(\sum_{j=1}^n A_{ij}^2\Big)^{p/2} \Big)^{1/p} \\
                                                             &=& \sqrt{p} \cdot \Big( \Big( \sum_{i=1}^m \Big( \sum_{j=1}^n B_{ij} \Big)^{p/2} \Big)^{2/p} \Big)^{1/2}  \\
                                                             &=& \sqrt{p} \cdot \Big( \Big\| \sum_{j=1}^n \bm{B}^j \Big\|_{p/2} \Big)^{1/2} \\ &\stackrel{\textrm{triangle ineq.}}{\leq}& \sqrt{p} \cdot \Big( \sum_{j=1}^n \underbrace{\|\bm{B}^j\|_{p/2}}_{\leq 1} \Big)^{1/2} \stackrel{(**)}{\leq} \sqrt{p} \cdot \sqrt{n}.
%                                                             &=& C\sqrt{p} \Big(\big\| \sum_{j=1}^n B^j \big\|_{p/2}\Big)^{1/2} \leq \sum \|B_i\|_1\Big)^{1/2} \\
%&\leq& C\sqrt{p} \Big(\sum_{j=1}^n \underbrace{\big\|\bm{B}^j\big\|_{p/2}}_{=\|A^j\|_p^2} \Big)^{1/2} \\
%                                                                 &\stackrel{???}{\leq}& C\sqrt{p}\Big(\sum_{j=1}^n \underbrace{\|A^j\|_p^2}_{\leq 1} \Big)^{1/2} \leq C\sqrt{p}\sqrt{n}
  \end{eqnarray*}
 Here we are using Jensen's Inequality (Lemma~\ref{lem:JensenForConcave}) with the concavity of $f(z) = z^{1/p}$ in $(*)$. 
  Finally in $(**)$ we have made use of
  \[
   \|\bm{B}^j\|_{p/2} = \Big(\sum_{i=1}^m (A_{ij}^2)^{p/2}\Big)^{2/p} = \Big( \underbrace{\sum_{i=1}^m |A_{ij}|^p}_{\leq 1} \Big)^{2/p} \leq 1.
  \]
  This shows the claim.
  % If $b_j := a_j^2$ then $\|a\|_2^p = (\sum_{j=1}^n a_j^2)^{p/2}$.
 % $b_i :$
\end{proof}

\subsection{A dimension-free upper bound on $\| \cdot \|_{\infty}$}
Now suppose we have a matrix $\bm{A} \in \setR^{m \times n}$ with $\bm{A}^1,\ldots,\bm{A}^n \in B_p^m$ and we want to
find a coloring $\bm{x}$ with small value of $\|\bm{A}\bm{x}\|_{\infty}$. If we were to take a random coloring
and the matrix happens to contain a row of all-ones, then in expectation one would have $\|\bm{A}\bm{x}\|_{\infty} \gtrsim \sqrt{n}$. It turns out that the Lovett-Meka distribution improves significantly over this bound. %, but converting the
However, obtaining a bound solely dependent on $n$ is slightly delicate since the Euclidean norm bound of $\|\bm{A}^j\|_2 \leq m^{1/2-1/p}$ on the columns
might be tight. %and we need to avoid a dependence on $m$.  
%bound on the columns into information on the it will b $\sqrt{n}$
%$L = \max\{ \|A^j\|_2 : j=1,ldots,n\} \leq $
\begin{lemma} \label{lem:LM-UpperBoundOnInfinityNormInTermsOfn}
  Let  $p \geq 2$, $\bm{x}_0 \in [-1,1]^n$ and let $\bm{A}^{m \times n}$ be a matrix with $\|\bm{A}^j\|_p \leq 1$ for all $j=1,\ldots,n$.
  Then for any $\bm{x} \sim \pazocal{D}(\bm{A},\bm{x}_0)$ one has $\|\bm{A}(\bm{x}-\bm{x}_0)\|_{\infty} \lesssim \sqrt{p} \cdot n^{1/2-1/p}$.
\end{lemma}
\begin{proof}
  The goal is to verify that the condition in Theorem~\ref{thm:LovettMekaRephrased}.(B) applies for a value of $\Delta$ that is of the order $\sqrt{p} \cdot n^{1/2-1/p}$. First we convert the bound on the $\| \cdot \|_p$-norm of the columns $\bm{A}^j$ into information about the $\| \cdot \|_2$-norm of the rows $\bm{A}_i$. In particular one has
  \[
  \Big(\frac{1}{n} \sum_{i=1}^m \|\bm{A}_i\|_2^p \Big)^{1/p} \stackrel{\textrm{Lem~\ref{lem:LpVsLq}}}{\leq} n^{1/2-1/p} \cdot \Big(\frac{1}{n} \sum_{i=1}^m \|\bm{A}_i\|_p^p\Big)^{1/p} = n^{1/2-1/p} \Big(\frac{1}{n} \sum_{j=1}^n\|\bm{A}^j\|_p^p\Big)^{1/p} \leq n^{1/2-1/p} \quad (***)
  \]
%  For a large enough constant $C>0$, we et $\Delta := C\sqrt{p} n^{1/2-1/p}$.
%  The goal is to show that the condition in Theorem~\ref{thm:LovettMekaRephrased}.(C) applies for this value of $\Delta$.
  Then we use this to estimate that
  \begin{eqnarray*}
  \sum_{i=1}^m \exp\Big(-C_1\frac{\Delta^2}{\|\bm{A}_i\|_2^2}\Big) &\leq& \sum_{i=1}^m p^{p/2} \frac{C_1^{p/2}\|\bm{A}_i\|_2^p}{\Delta^p} \leq n \cdot \Big(\frac{\sqrt{C_1p} n^{1/2-1/p}}{\Delta}\Big)^p \leq C_2n
  \end{eqnarray*}
  if we set $\Delta := \frac{\sqrt{C_1}}{C_2^{1/p}} \sqrt{p} \cdot n^{1/2-1/p}$.
  Here we have used the following estimate:  \\
  {\bf Claim I.} \emph{For $p \geq 1$ and $y>0$ one has $\exp(-\frac{1}{y}) \leq p^{p/2}y^{p/2}$.} \\
  Indeed, for $y > 1$, one has $\exp(-\frac{1}{y}) < 1 < p^{p/2}y^{p/2}$. Since $\exp(\frac{1}{y}) = \sum_{k \in \setN} \frac{1}{k! \cdot y^k} \ge \frac{1}{k! \cdot y^k}$ for any $k \in \setN$, one has for $y \le 1$: 
  \[\exp\Big(-\frac{1}{y}\Big) \le \lceil p/2\rceil! \cdot y^{\lceil p/2\rceil}  \le  \lceil p/2\rceil^{\lfloor p/2 \rfloor} y^{p/2} \le p^{p/2} y^{p/2}. \qedhere \]
\end{proof}

\subsection{Balancing from $B_p^m$ to $B_q^m$}

Now we show how to find partial colorings for balancing vectors in the $B_p^m$-ball
into scalars of the  $B_q^m$-ball using the Lovett-Meka distribution: %prove the first part of Theorem 1 in the paper.
\begin{theorem} \label{thm:LqBoundForLpBoundedColumns}
  Let $2 \leq p \leq q \leq \infty$, $\bm{x}_0 \in [-1,1]^n$ and let $\bm{A} \in \setR^{m \times n}$ with $\bm{A}^1,\ldots,\bm{A}^n \in B_p^m$.
  Then
  \[
    \E_{\bm{x} \sim \pazocal{D}(\bm{A},\bm{x}_0)}\big[\|\bm{A}(\bm{x}-\bm{x}_0)\|_q\big] \lesssim \sqrt{\min\big\{ p, \log(\tfrac{2m}{n})\big\} } \cdot n^{1/2 - 1/p + 1/q}
  \]
\end{theorem}
\begin{proof}
  We will make use of the inequality $\|\bm{z}\|_q \leq (\|\bm{z}\|_p^p \cdot \|\bm{z}\|_{\infty}^{q-p})^{1/q}$ for all $\bm{z} \in \setR^m$,
  see Lemma~\ref{lem:LqVsLpAndLinfinity}.
  Then combining the estimates from Lemma~\ref{lem:LM-ExpectedLpNorm} and Lemma~\ref{lem:LM-UpperBoundOnInfinityNormInTermsOfn} we obtain
  \begin{eqnarray*}
   \E_{\bm{x} \sim \pazocal{D}(\bm{A},\bm{x}_0)}[\|\bm{A}(\bm{x}-\bm{x}_0)\|_q] &\leq& \E_{\bm{x} \sim \pazocal{D}(\bm{A},\bm{x}_0)}\big[ \big(\|\bm{A}(\bm{x}-\bm{x}_0)\|_p^p \cdot \|\bm{A}(\bm{x}-\bm{x}_0)\|_{\infty}^{q-p}\big)^{1/q}\big]  \quad \quad (****) \\
                                                            &\stackrel{\textrm{Lem~\ref{lem:LM-UpperBoundOnInfinityNormInTermsOfn}}}{\lesssim}& \E_{\bm{x} \sim \pazocal{D}(\bm{A},\bm{x}_0)}\big[ \|\bm{A}(\bm{x}-\bm{x}_0)\|_p^{p/q}\big] \cdot (\sqrt{p} n^{1/2-1/p})^{(q-p)/q} \\
    &\stackrel{\textrm{Jensen}}{\leq}& \E_{\bm{x} \sim \pazocal{D}(\bm{A},\bm{x}_0)}\big[ \|\bm{A}(\bm{x}-\bm{x}_0)\|_p\big]^{p/q} \cdot (\sqrt{p} n^{1/2-1/p})^{(q-p)/q} \\
     &\stackrel{\textrm{Lem~\ref{lem:LM-ExpectedLpNorm}}}{\lesssim}& (\sqrt{pn})^{p/q}  \cdot (\sqrt{p} n^{1/2-1/p})^{(q-p)/q}
                  = \sqrt{p n} \cdot n^{-1/p + 1/q}
  \end{eqnarray*}
  Here we have used Jensen's inequality (see Lemma~\ref{lem:JensenForConcave}) with the concavity of the function $y \mapsto y^{p/q}$ for $0<y<\infty$.

  Note that this settles the claim in the parameter range  $p \leq \ln(e^2\frac{m}{n})$. Now consider the
  case of  $p > \ln(e^2\frac{m}{n})$. We define $p_0 := \ln(e^2\frac{m}{n})$. Note that $2 \leq p_0 \leq p$.
  Then applying the bound of $(****)$ proven above, with parameters $2 \leq p_0 \leq q$ we
  obtain %that $\Pr_{\bm{x} \sim \pazocal{D}(\bm{A})}[\|\bm{A}\bm{x}\|_q \leq \Delta] \geq \frac{1}{2}$ where
  \begin{eqnarray*}
    \Pr_{\bm{x} \sim \pazocal{D}(\bm{A},\bm{x}_0)}\big[\|\bm{A}(\bm{x}-\bm{x}_0)\|_q\big] &\lesssim& \sqrt{p_0} \cdot n^{1/2 - 1/p_0 + 1/q} \cdot \max\{ \|\bm{A}^j\|_{p_0} : j=1\ldots,n\} \\
           &\leq& \sqrt{p_0} \cdot n^{1/2-1/p+1/q} \cdot m^{1/p_0-1/p} \cdot \max\{ \underbrace{\|\bm{A}^j\|_p}_{\leq 1} : j=1,\ldots,n\} \\
           &=& \sqrt{p_0} \cdot n^{1/2-1/p+1/q} \cdot \underbrace{\Big(\frac{m}{n}\Big)^{-1/p}}_{\leq 1} \cdot \Big(\frac{m}{n}\Big)^{1/p_0} \\
    &\leq& \sqrt{\ln\Big(e^2\frac{m}{n}\Big)} \cdot n^{1/2-1/p+1/q} \cdot \underbrace{\Big(\frac{m}{n}\Big)^{1 / \ln(e^2\frac{m}{n})}}_{\leq O(1)}
  \end{eqnarray*}
  This completes the proof.
\end{proof}

\subsection{From partial colorings to full colorings}

The result from Theorem~\ref{thm:LqBoundForLpBoundedColumns} allows us to find a partial coloring --- the next step is to iterate the argument in order to find a full coloring:
\begin{proof}[Proof of Theorem~\ref{thm:VectorBalancingFromBpToBq}]
  Consider $n$ vectors from $B_p^m$ that we conveniently write as column vectors $\bm{A}^1,\ldots,\bm{A}^n \in B_p^m$. We set $\bm{x}_0 := \bm{0} \in \setR^n$ and $n_0 := n$. In iteration $t=0,1,\ldots$ we
  have maintained a vector $\bm{x}_t \in [-1,1]^n$ where $n_t := |\{ i \in [n] \mid -1<x_t(i)<1\}|$
  denotes the number of uncolored elements. We draw $\bm{x}_{t+1} \sim \pazocal{D}(\bm{A},\bm{x}_t)$ and repeat until
  the bound of the expectation provided by Theorem~\ref{thm:LqBoundForLpBoundedColumns} is
  attained (say up to a factor of 2). In particular
  \[
    \big\|\bm{A}(\bm{x}_{t+1}-\bm{x}_t)\big\|_q \lesssim \sqrt{\min\big\{ p,\log\big(\tfrac{2m}{n_t}\big)\big\} } \cdot n_t^{1/2-1/p+1/q}
  \]
  and $n_{t+1} \leq n_{t}/2$. Then $n_{t+1} \leq \frac{n}{2^t}$ and so $\bm{x}^* := \bm{x}_{\log_2(n)+1}$
  will be in $\{ -1,1\}^n$ and by the triangle inequality, the discrepancy is at most
  \[
    \|\bm{A}\bm{x}^*\|_q \lesssim \sum_{t \geq 0} \sqrt{\min\big\{ p,\log\big(\tfrac{2m}{n/2^t}\big)\big\} } \cdot (n/2^t)^{1/2-1/p+1/q} \lesssim \frac{\sqrt{\min\{ p,\log(\frac{2m}{n})\} }}{1/2-1/p+1/q}  n^{1/2-1/p+1/q}
  \]
  To see the ultimate inequality, consider the exponent $\alpha := 1/2-1/p+1/q>0$. Then it takes $1/\alpha$ iterations until the quantity $n_t^{\alpha}$ has decreased by a factor of 2 (while the term $\log(2m/n_t)$ has a miniscule growth). Then the cumulated discrepancy is dominated by the first $\frac{1}{\alpha}$ terms. The running time is bounded by $O(\log n) \cdot T(m,n)$.
\end{proof}

\section{Approximate Carath\'eodory bounds for $\| \cdot \|_p$ norms}
Next, we prove the bound on $\textrm{ac}_k(B_p^m,B_q^m)$ claimed in Theorem~\ref{ApxCaraForLpLq}: 
\begin{proof}[Proof of Theorem~\ref{ApxCaraForLpLq}]
 % It remains to note that the first term in the sum in Theorem~\ref{thm:ReductionApxCaraToVB} dominates the remaining terms for $2 \le p \le q \le \infty$ since they are geometrically decreasing:
  Let  $2 \le p \le q \le \infty$. We apply the reduction to the vector balancing constant from Theorem~\ref{thm:ReductionApxCaraToVB} and combine this with the bound from Theorem~\ref{thm:VectorBalancingFromBpToBq}:
  \begin{eqnarray*}
    \textrm{ac}_k(B_p^m,B_q^m) &\stackrel{\textrm{Thm~\ref{thm:ReductionApxCaraToVB}}}{\leq}& \sum_{\ell \geq 1} \frac{1}{k2^{\ell}} \cdot \textrm{vb}_{k2^{\ell}}(B^m_p, B^m_q) \\
                               &\stackrel{\textrm{Thm~\ref{thm:VectorBalancingFromBpToBq}}}{\lesssim}& \frac{1}{\frac{1}{2}-\frac{1}{p}+\frac{1}{q}}\sum_{\ell \geq 1} \frac{\sqrt{\min\{ p, \log (\frac{2m}{k 2^\ell}) \}}}{(k 2^\ell)^{1/2+1/p-1/q}} \\ & \lesssim& \frac{1}{\frac{1}{2}-\frac{1}{p}+\frac{1}{q}} \cdot \frac{\sqrt{\min\{ p, \log (\frac{2m}{k}) \}}}{k^{1/2+1/p-1/q}}.  \end{eqnarray*}
                             Note that the exponent $\alpha := 1/2+1/p-1/q$ satisfies $\alpha \geq 1/2$ and so the sum is already dominated by the very first term. The running time to find the vectors $\bm{v}_1,\ldots,\bm{v}_k$ is dominated by $O(\log m)$ calls to find the coloring behind $\textrm{vb}_{t}(B_p^m,B_q^m)$ where $t \leq m+1$ which results in a total running time
                             of $O(\log^2 m) \cdot T(m,m+1) \leq O(m^{1+\omega}\log^2(m))$ as $\omega \geq 2$.
\end{proof}

\begin{remark}
In the case where $p = 2$ and $q = \infty$, we can apply Theorem~\ref{ApxCaraForLpLq} with $q' := \log_2 m$ to obtain $\textrm{ac}_k(B_2^m, B_\infty^m) \lesssim \frac{\log m}{k}$ by noting that Lemma~\ref{lem:LpVsLq} implies $\|\bm{x}\|_\infty \le \|\bm{x}\|_{\log_2 m} \le 2 \|\bm{x}\|_\infty$ for any $\bm{x} \in \setR^m$. But for $P = B^m_2$ and $Q = B^m_\infty$ one has $\textrm{vb}_k (B^m_2, B^m_\infty) \lesssim \sqrt{\log m}$ by the result of Banaszczyk~\cite{BalancingVectors-Banaszczyk98}, so that Theorem~\ref{thm:ReductionApxCaraToVB} yields the improved bound \[\textrm{ac}_k(B_2^m, B_\infty^m) \lesssim \frac{\sqrt{\log m}}{k}.\]
It remains an interesting open question whether this may be improved to $O(\frac{1}{k})$.
\end{remark}

\section{Lower bounds for vector balancing}

In this section we show that the vector balancing bounds in Theorem~\ref{thm:VectorBalancingFromBpToBq} are tight up to the factor of $\frac{1}{2} - \frac{1}{p} + \frac{1}{q}$:

\begin{proof}[Proof of Theorem~\ref{thm:VBLowerBound}]
First, let us focus on the case when $m \le 2^p n$ so that $\log(2m/n) = \min\Big\{ p, \log \Big(\frac{2m}{n}\Big) \Big\}$. We will also assume that $m \ge 8n$, since otherwise we can use instead $n' := \lfloor n/8 \rfloor$ and add $n-n'$ columns of zeros.
 Let $\bm{B}$ denote an $m \times n$ random matrix with i.i.d $\pm 1$ entries. For any fixed $\bm{x} \in \{-1,1\}^n$, let $N_{\bm{x}}$ denote the number of rows $i \in [m]$ with $\inn{\bm{B}_i}{\bm{x}} \ge \lambda \sqrt{n}$ for $\lambda :=  \sqrt{\frac{2}{9} \log\Big(\frac{m}{2n}\Big)}$. Since $\lambda \le \sqrt{n}/2$, by Lemma~\ref{lem:AntiChernoff} we have  
\[\Pr\Big[\inn{\bm{B}_i}{\bm{x}} \ge \lambda \sqrt{n}\Big] \ge \frac{2n}{m},\] 
so that $\E[N_{\bm{x}}] \ge 2n$. The standard Chernoff bound then gives $\Pr[N_{\bm{x}} \le (1-0.9) \cdot 2n] < 2^{-n}$, so that by the union bound there exists a matrix $\bm{B} \in \{-1,1\}^{m \times n}$ for which $N_{\bm{x}} \ge n/5$ for all $\bm{x} \in \{-1,1\}^n$. Thus for any such $\bm{x}$,
\[ \|\bm{Bx}\|_q \ge (|N_{\bm{x}}| \cdot (\lambda \sqrt{n})^q)^{1/q} \gtrsim n^{1/2+1/q} \log\Big(\frac{m}{2n}\Big)^{1/2} \gtrsim n^{1/2+1/q} \log\Big(\frac{2m}{n}\Big)^{1/2},\]
where in the last step we used $m \ge 8n$. The matrix $\bm{A} := m^{-1/p} \bm{B}$ has columns in $B^m_p$ and 
\[\|\bm{Ax}\|_q \gtrsim \sqrt{\log\Big(\frac{m}{2n}\Big)} \cdot \frac{n^{1/2+1/q}}{m^{1/p}} \gtrsim \sqrt{\log\Big(\frac{m}{2n}\Big)}\cdot  n^{1/2-1/p+1/q}\] for all $\bm{x} \in \{-1,1\}^n$, as claimed. 

For the case when $m \ge 2^p n$ we may use the same construction for $m' := \lfloor 2^p n \rfloor$ with $m-m'$ additional rows of zeros.
 \end{proof}

\bibliographystyle{alpha}
\bibliography{approxCaratheodory}

\appendix

\section{The Lovett-Meka random walk\label{appendix:LovettMeka}}

In this section we want to justify Theorem~\ref{thm:LovettMekaRephrased}. %the modification
First, the original main technical statement of Lovett and Meka~\cite{DiscrepancyMinimization-LovettMekaFOCS12} is as follows:
\begin{theorem}[\cite{DiscrepancyMinimization-LovettMekaFOCS12}] \label{thm:LovettMekaOriginal}
  Let $\bm{A}_1,\ldots,\bm{A}_m \in \setR^n$, $\bm{x}_0 \in [-1,1]^n$, $\delta>0$ small enough and let $\lambda_1,\ldots,\lambda_m \geq 0$ so that $\sum_{i=1}^m \exp(-\lambda_i^2/16) \leq \frac{n}{16}$. Then there exists an efficient randomized algorithm which with probability at least $\frac{1}{10}$ finds a point $\bm{x} \in [-1,1]^n$ so that
  \begin{enumerate}
  \item[(i)] $|\left<\bm{A}_i,\bm{x}-\bm{x}_0\right>| \leq \lambda_i\|\bm{A}_i\|_2$ for all $i=1,\ldots,m$.
  \item[(ii)] $|x_j| \geq 1-\delta$ for at least $n/2$ many indices.
  \end{enumerate}
  Moreover, the algorithm runs in time $O( (m+n)^3 \cdot \delta^{-2} \log(mn/\delta))$.
\end{theorem}
First, note that if $\Delta$ is a  parameter with $\sum_{i=1}^m \exp(-\frac{\Delta^2}{16\|\bm{A}_i\|_2^2}) \leq \frac{n}{16}$,
then $\lambda_i := \frac{\Delta}{\|\bm{A}_i\|_2}$ is a feasible choice. % for Theorem~\ref{thm:LovettMekaOriginal}.
The algorithm behind Theorem~\ref{thm:LovettMekaOriginal} works as follows: let $\gamma>0$ be a parameter that is logarithmically smaller than $\delta$. We compute a random sequence $\bm{x}_0,\bm{x}_1,\bm{x}_2,\ldots$ of points.
In iteration $t$ we have already computed $\bm{x}_t$; next we define a subspace $V_t \subseteq \setR^n$ that is orthogonal to coordinates $\bm{e}_i$ with $|x_i| \geq 1-\delta$ and to directions $\bm{A}_i$ with $|\left<\bm{A}_i,\bm{x}-\bm{x}_0\right>| \geq \lambda_i-\delta$. Then update $\bm{x}_{t+1} = \bm{x}_t + \gamma \bm{u}_t$ where $\bm{u}_t \sim N(\bm{0},\bm{I}_{V_t})$ is a random Gaussian from that subspace $V_t$.

What is slightly unsatisfactory is that we will not actually have coordinates equal to $\pm 1$.
But there is an easy way to remedy this: consider the polyhedron $K := \{ \bm{x} \in [-1,1]^n \mid |\left<\bm{A}_i,\bm{x}-\bm{x}_0\right>| \leq \lambda_i\|\bm{A}_i\|_2 \;\; \forall i=1,\ldots,m\}$.
We say that a coordinate $j$ is \emph{tight} for a point $\bm{x}$
if $|x_j| = 1$ and a constraint $i$ is \emph{tight} for $\bm{x}$ if $|\left<\bm{A}_i,\bm{x}-\bm{x}_0\right>| = \lambda_i\|\bm{A}_i\|_2$.
Then consider the following random walk:
\begin{enumerate}
\item[(1)] FOR $t=1,2,\ldots,8n$ DO
  \begin{enumerate}
  \item[(2)] Let $V_t := \{ \bm{y} \in \setR^n \mid \left<\bm{y},\bm{x}_t\right>=0; \; y_j=0\textrm{ if }j \in [n]\textrm{ is tight for }\bm{x}_t; \; \left<\bm{A}_i,\bm{y}\right> = 0\textrm{ if }i\textrm{ is tight for }\bm{x}_t\}$. If $\dim(V_t) \leq \frac{n}{8}$ then return FAIL
  \item[(3)] Let $\bm{u}_t \sim S^{n-1} \cap V_t$.
  \item[(4)] Let $r_t := \max\{ r \geq 0 \mid \bm{x}_t + r \bm{u}_t \in K\textrm{ and }\bm{x}_t - r\bm{u}_t \in K\}$
  \item[(5)] Let $\delta_t := \min\{ 1,r_t\}$
  \item[(6)] Update $\bm{x}_{t+1} := \bm{x}_t + \sigma_t\delta_t \bm{u}_t$ where $\sigma_t \sim \{ -1,1\}$
  \item[(7)] IF $\bm{x}_{t+1}$ has at least $n/2$ coordinates tight THEN return $\bm{x}_{t+1}$. 
  \end{enumerate}
  \item[(8)] RETURN FAIL
  \end{enumerate}

  The key modification compared to \cite{DiscrepancyMinimization-LovettMekaFOCS12} is that the
  step length is not uniform but it is capped so that we do not exit $K$.
  As we walk orthogonal to the current point, one has $\|\bm{x}_{t+1}\|_2^2 = \|\bm{x}_t\|_2^2 + \delta_t^2$ for all $t$
  and so $\sum_t \delta_t^2 \leq 4n$. Let us call a step $t$ \emph{long} if $\delta_t = 1$
  and \emph{short} otherwise. So we cannot have more than $4n$ many long steps. But in each short step (i.e. $\delta_r = r_t$)
  we have a probability of at least 1/2 that some coordinate or constraint becomes tight
  and this cannot happen more than $\frac{7}{8}n$ times before exiting at (2).
  So the probability of reaching $(8)$ is less than $(1/2)^n$.
  
  Now let $\bm{x}$ be the random vector when the algorithm stops (either in (2), (7) or (8)).
  Fix a constraint $i$; the next step is to analyze the concentration behaviour of $\left<\bm{A}_i,\bm{x}-\bm{x}_0\right>$.
  In each iteration $t$ one has $\|\left<\bm{A}_i,\sigma_t\delta_t\bm{u}_t\right>\|_{\psi_2} \leq \|\left<\bm{A}_i,\bm{u}_t\right>\|_{\psi_2} \lesssim \frac{\|\bm{A}_i\|_2}{\sqrt{n}}$ by Lemma~\ref{lem:PropertiesOfSubgaussianNorm} as $\bm{u}_t$
  is a unit vector from a subspace of dimension at least $n/8$. 
%  Then the deviation $|<A_i,x-x_0>|$ comes from taking at most $8n$
%  steps into a (at most) unit direction sampled from subspace of dimension at least $n/8$.
  As there are at most $8n$ iterations, we have $\|\left<\bm{A}_i,\bm{x}-\bm{x}_0\right>\|_{\psi_2} \lesssim 8n \cdot (\frac{\|\bm{A}_i\|_2}{\sqrt{n}})^2 \lesssim \|\bm{A}_i\|_2$.
  Then $\Pr[|\left<\bm{A}_i,\bm{x}_0-\bm{x}\right>| \geq \lambda_i\|\bm{A}_i\|_2] \leq \exp(-C_1\lambda_i^2)$ by Lemma~\ref{lem:PropertiesOfSubgaussianNorm}.(D) for some constant  $C_1>0$. In other words, the probability that the $i$th constraint becomes tight at any point in the algorithm is upper bounded by $\exp(-C_1\lambda_i^2)$
  and so with good probability the number of tight constraints remains below, say, $\frac{n}{4}$.
  This gives an upper bound on the probability to exit in (2). Finally,
  we define $\pazocal{D}(\bm{A},\bm{x}_0)$ as the vector $\bm{x}_{t+1}$ returned in (7), that means
  conditioned on the run being successful.
  
  We briefly discuss the running time. In every of the at most $8n$ iterations we compute a basis of
  $V_t$ which using fast matrix multiplication can be done in time $O(n^{\omega})$. Then
  to compute $r_t$ (and to determine which constraints determine $V_t$) can be done in time $O(nm)$.
  This results in the claimed bound of $T(m,n) \leq O(n^{1+\omega} + n^2m)$.

\end{document}